\documentclass[12pt]{amsart}

\usepackage{amssymb,amscd,amsthm,verbatim,amsmath,color,fancyhdr,mathrsfs,bm,color}
\usepackage{graphicx}
\usepackage{turnstile}

\numberwithin{equation}{section}

\usepackage[letterpaper, left=2.5cm, right=2.5cm, top=2.5cm,bottom=2.5cm,dvips]{geometry}

\usepackage{mathtools}
\usepackage{bbm}
\usepackage[colorlinks = true, urlcolor={blue}, citecolor={blue}]{hyperref}
\usepackage{etoolbox}
\usepackage{multicol}
\usepackage{url,stackengine,pifont,multirow,caption,subcaption}
\usepackage{paralist} 
\usepackage{enumitem}

\usepackage{neuralnetwork}
\usepackage{tikz,tikz-cd,pgfplots}
\usetikzlibrary{cd,calc,intersections,decorations.markings,arrows}
\pgfplotsset{compat=newest}

\newcommand{\E}{\mathcal{E}}
\newcommand{\1}{\textbf{1}}
\newcommand{\ep}{\epsilon}
\newcommand{\G}{\mathcal{G}}

\newtheorem{theorem}{Theorem}[section]
\newtheorem{proposition}[theorem]{Proposition}

\newtheorem{remark}[theorem]{Remark}
\newtheorem{corollary}[theorem]{Corollary}

\newtheorem{example}[theorem]{Example}

\DeclarePairedDelimiterX{\bracket}[3]{#1}{#2}{#3}

\providecommand{\newoperator}[3]{\newcommand*{#1}{\mathop{#2}#3}}
\providecommand{\renewoperator}[3]{\renewcommand*{#1}{\mathop{#2}#3}}

\renewoperator{\Re}{\mathrm{Re}}{\nolimits}
\renewoperator{\Im}{\mathrm{Im}}{\nolimits}

\DeclarePairedDelimiterXPP{\nrm}[2]{}{\lVert}{\rVert}{\ensuremath{_{#1}}}{\ifblank{#2}{\:\cdot\:}{#2}}

\DeclarePairedDelimiterXPP\prob[1]{\mathbb{P}}{\lbrace}{\rbrace}{}{#1} 

\DeclarePairedDelimiterXPP\probability[2]{\mathbb{P}_{#1}}{\lbrace}{\rbrace}{}{#2} 

\DeclarePairedDelimiterXPP\expectation[1]{\mathbb{E}}{\lbrack}{\rbrack}{}{#1} 

\DeclarePairedDelimiterXPP\expectationdist[2]{\mathbb{E}_{#1}}{\lbrack}{\rbrack}{}{#2} 

\DeclarePairedDelimiterXPP\variance[1]{\mathrm{Var}}{\lbrack}{\rbrack}{}{#1} 

\DeclarePairedDelimiterXPP\variancedist[2]{\mathrm{Var}_{#1}}{\lbrack}{\rbrack}{}{#2} 

\DeclarePairedDelimiterXPP\covariance[2]{\mathrm{Cov}}{(}{)}{}{#1,\mathopen{}#2} 

\newoperator{\supp}{\mathrm{supp}}{\nolimits}

\makeatletter
\providecommand*{\diff}
{\@ifnextchar^{\DIfF}{\DIfF^{}}}
\def\DIfF^#1{
	\mathop{\mathrm{\mathstrut d}}
	\nolimits^{#1}\gobblespace}
\def\gobblespace{
	\futurelet\diffarg\opspace}
\def\opspace{
	\let\DiffSpace\!
	\ifx\diffarg(
	\let\DiffSpace\relax
	\else
	\ifx\diffarg[
	\let\DiffSpace\relax
	\else
	\ifx\diffarg\{
	\let\DiffSpace\relax
	\fi\fi\fi\DiffSpace}

\providecommand*{\pdiff}
{\@ifnextchar^{\pDIfF}{\pDIfF^{}}}
\def\pDIfF^#1{
	\mathop{\mathrm{\mathstrut \partial}}
	\nolimits^{#1}\gobblespace}
\def\gobblespace{
	\futurelet\diffarg\opspace}
\def\opspace{
	\let\DiffSpace\!
	\ifx\diffarg(
	\let\DiffSpace\relax
	\else
	\ifx\diffarg[
	\let\DiffSpace\relax
	\else
	\ifx\diffarg\{
	\let\DiffSpace\relax
	\fi\fi\fi\DiffSpace}

\makeatother

\DeclarePairedDelimiterX\Set[1]\{\}{
	
	#1
}

\newcommand{\N}{\mathbf{N}}

\newcommand{\Gcal}{\mathcal{G}}

\newcommand{\eq}{\begin{equation}}
\newcommand{\en}{\end{equation}}

\usepackage{amsrefs} 

\title{Small noise limits of Markov chains and the PageRank}

\author[Borkar]{Vivek S.\ Borkar}
\address{Department of Electrical Engineering \\ Indian Institute of Technology Bombay\\ Powai, Mumbai 400076, India\\ {Email: borkar.vs@gmail.com}}

\author[Sowmya]{S. Sowmya}
\address{Department of Electrical Communication Engineering \\ Indian Institute of Science\\ Bengaluru 560012, India\\ {Email: ssowmya@iisc.ac.in}}

\author[Tripathi]{Raghavendra Tripathi}
\address{Department of Mathematics\\ Division of Science \\ New York University\\ Abu Dhabi, UAE\\ {Email: rt1986@nyu.edu}}

	

\begin{document}

\begin{abstract} 
We recall the classical formulation of  PageRank as the stationary distribution of a singularly perturbed irreducible Markov chain that is not irreducible when the perturbation parameter goes to zero. Specifically, we use the Markov chain tree theorem to derive explicit expressions for the PageRank. This analysis leads to some surprising results. These results are then extended to a much more general class of perturbations that subsume personalized PageRank. We also give examples where even simpler formulas for PageRank are possible.
\end{abstract}

\maketitle
\section{Introduction}\label{Introd}

Google's PageRank algorithm for ranking web pages \cite{Brin} has been an extensively studied topic since its inception. The early accounts \cite{Lang1}, \cite{Lang2}, though somewhat dated, give an excellent exposition. See \cite{Berkhin}, \cite{ChungSurvey}, \cite{chung2010}, \cite{Park},\cite{Yang} for more recent surveys. There are also other related research directions such as dynamic PageRank \cite{Soumen1}, \cite{Sal} and PageRank on random graphs \cite{Chen}, \cite{Lee}. See also \cite{Cominetti} for a game theoretic formulation of PageRank where the Markov chain tree theorem plays a role as it does here. The literature on PageRank is immense and rich, particularly the schemes that build upon the basic framework of raw PageRank as originally defined. See \cite{Alekh}, \cite{Soumen}, \cite{Haveli}, \cite{Hend}, \cite{Theobald} for other research directions.

In this note, we address a novel issue, viz.\ the PageRank viewed with the lens of small noise limits for stationary distributions of singularly perturbed Markov chains. This leads to additional insights about PageRank, sometimes surprising. 

In the next section, we recall a key result from Markov chain theory, the Markov chain tree theorem, that facilitates our approach. In section 3, we present our main results. Section 4 gives a few concrete examples. Section 5 concludes with a brief discussion.

\section{Markov chain preliminaries}
\label{sec:Markov}

Here we recall some standard facts about finite-state Markov chains. Some good general references for this (and more) are \cite{Haggstrom}, \cite{Norris}, \cite{Wilmer}. A Markov chain $(X_n)_{n\geq 0}$ on a finite state space $S:= \{1,2,\cdots, N\}$ is described by an $n\times n$ transition matrix $P$ such that 
$$P(i, j) = \mathbb{P}(X_{1}= j\vert X_{0}=i) = \mathbb{P}(X_{n+1}= j\vert X_{n}=i, X_m,m < n).$$
Let $P^{n}$ denote the $n$-fold matrix product of $P$. It is easy to see that $P^n(i, j)= \mathbb{P}(X_n=j\vert X_0=i)$. Two states $i$ and $j$ are said to \emph{communicate} if $P^{n}(i, j)>0$ and $P^{m}(j, i)>0$ for some $n, m\in \mathbb{N}$. It can be easily verified that communication is an equivalence relation. In particular, it partitions the state space into the so-called \emph{communicating classes} plus  possibly some \textit{transient} states, i.e.\ states that do not belong to any communicating class. We say that the chain is \textit{reducible} when it has more than one communicating class. Otherwise we say that it is \textit{irreducible}. 

Let $P$ be the $N\times N$ transition matrix of a (possibly reducible) Markov chain on a finite state space $S:= \{1,2,\cdots, N\}$. When $P$ is irreducible, it has a unique stationary distribution. That is, there exists a unique probability measure $\pi$ on $S$ (treated as a row vector) such that $\pi P=\pi$. The situation, however, is different for a reducible Markov chain. In particular, if $P$ is reducible, then each closed communicating class supports a unique stationary distribution. These stationary distributions, after extending them to the whole state space by assigning zero probability to the remaining states, form the extreme points of the polytope of all stationary distributions. In particular, no stationary distribution assigns positive probability to a transient state. In the absence of a unique stationary distribution, a commonly employed technique for `selecting' a stationary distribution is via the zero noise limit. The idea is as follows. Take an irreducible transition matrix $Q = [[q(i,j)]]_{i,j\in S}$ on the same state space and let $P_{\epsilon}\coloneqq (1-\epsilon)P+\epsilon Q$. The choice of $Q$ is usually dictated by the application. We shall see some examples of $P$ and $Q$ later. A standard choice for $Q$ is $\frac{1}{N}\1\1^{T}$. For $\epsilon>0$, $P^\epsilon$ defines an irreducible transition matrix and hence admits a unique stationary distribution $\pi^{\epsilon}$. Letting $\epsilon \to 0$ in the equation
$$\pi^\epsilon  = \pi^\epsilon P_{\epsilon}, \ \sum_{i\in S}\pi^\epsilon(i) = 1,$$ it is easy to see that any limit point of $\pi_{\epsilon}$ as $\epsilon\to 0$ is a stationary distribution of $P$. In some cases in applications, one can show that $\pi_{\epsilon}$ has a unique limit as $\epsilon\to 0$, and this limiting distribution is then viewed as the `physical' stationary distribution of $P$. This is called the small noise limit~\cite{freidlin1998random}. In general, the limit may be non-unique and may comprise of all possible stationary distributions of $P$. We encounter the latter situation later in this work.

Next, we describe the Markov chain tree theorem (see e.g.~\cite{Shubert}) that gives an expression for the stationary distribution of an irreducible Markov chain in terms of the weights of its arborescences. Our treatment of the Markov chain tree theorem follows~\cite{Ananth}. Consider an irreducible Markov chain on $S$ with transition matrix $\widetilde{P}$. Define an \emph{arborescence} as a directed spanning tree with at most one outgoing edge from each node. Then there is exactly one node in this tree with no outgoing edge, dubbed the `root' of the arborescence. The weight of an arborescence $A$, denoted $|A|$, is the product of transition probabilities of the directed edges in $A$. For $i \in S$, let $H(i)$ denote the set of arborescence rooted in $i$, with its weight defined as $|H(i)| := \sum_{A\in H(i)}|A|$. The Markov chain tree theorem states the following.

\begin{theorem}[Markov chain tree theorem] Let $\widetilde{\pi}$ be the unique stationary distribution of an irreducible Markov chain on the finite state space $S=\{1, \ldots, n\}$ with transition matrix $\widetilde{P}$. Then, 
$\widetilde{\pi}(i) = \frac{|H(i)|}{\sum_j|H(j)|}$ for $i\in S$.
\end{theorem}

In section 3, we use the Markov chain tree theorem to analyze the zero noise limit of the PageRank algorithm. We state our main result here. 

\begin{theorem}
\label{thm:Main}
Let $P$ be a reducible Markov chain on finite state space $S=\{1, \ldots, n\}$. Let $C_1,  \ldots, C_m$ be the closed communicating classes of $P$.  For $k\in [m]$, let $\pi_k$ denote the unique stationary distribution of $P$ supported on $C_k$.  
Let $P_{\epsilon}\coloneqq (1-\epsilon)P+\epsilon \frac{1}{n}\1\1^{T}$ and let $\pi^{(\epsilon)}$ denote the unique stationary distribution of $P_{\epsilon}$. Then, for every $k\in [m]$ and $v\in C_k$, we have 
\[\lim_{\epsilon\to 0}\pi^{(\epsilon)}(v) = \frac{\pi_{k}(v)}{m}\;.\]
In particular, we have 
\[\lim_{\epsilon\to 0}\pi^{(\epsilon)}(C_k) = \frac{1}{m}\;,\]
for every $k\in [m]$. 
\end{theorem}

In the context of the PageRank algorithm, the matrix $P$ represents the transition matrix of the simple random walk on the web graph. The basic idea behind the PageRank algorithm is to rank the web pages according to the stationary distribution of $P$. However,  the web graph has some small clusters and even some dangling nodes (i.e., absorbing states, equivalently, the states with no outgoing edge pointing to a different state) and therefore it has no unique stationary distribution. This necessitates the need for some regularization. As explained in the introduction, this is precisely the issue addressed by the perturbaed transition matrix $P_\ep$. But then it calls for studying its zero-noise limit which the stationary distribution for small $\ep$ will closely approximate. However, as Theorem~\ref{thm:Main} shows, with the zero noise limit, some very small clusters (even a dangling node) get the same weight as a large cluster. This naturally raises the following question: Can we choose the regularizing matrix $Q$ in the above theorem to obtain a target measure $\pi$ as the zero noise limit of $P$? Towards this problem, we first describe the zero-noise limit for an arbitrary regularizing matrix $Q$ in the following theorem. 
\begin{theorem}
\label{thm:general}
Let $P$ be a reducible Markov chain on finite state space $S=\{1, \ldots, n\}$. Let $C_1,  \ldots, C_m$ be the closed communicating classes of $P$.  For $k\in [m]$, let $\pi_k$ denote the unique stationary distribution of $P$ supported on $C_k$.  Let $Q$ be an $n\times n$ stochastic matrix and let $\Gamma\coloneqq \Gamma_{Q}$ be an $m\times m$ stochastic matrix such that $\Gamma(i, j)=\frac{1}{|C_i|}\sum_{x\in C_i, y\in C_j}Q(x, y)$. Assume that $\Gamma$ is irreducible and let $\pi_{\Gamma}$ be the unique stationary distribution of $\Gamma$. Let $P_{\epsilon}\coloneqq (1-\epsilon)P+\epsilon Q$ and let $\pi^{(\epsilon)}$ be the unique stationary distribution of $P_{\epsilon}$. 
Then, for every $k\in [m]$ and $v\in C_k$, we have 
\[\lim_{\epsilon\to 0}\pi^{(\epsilon)}(v) = \pi_{k}(v)\pi_{\Gamma}(k)\;.\]
In particular, we have 
\[\lim_{\epsilon\to 0}\pi^{(\epsilon)}(C_k) = \pi_{\Gamma}(k)\;,\]
for every $k\in [m]$. 
\end{theorem}  
 
 \begin{remark} Some remarks are in order.
 \begin{enumerate}
 \item It is easy to see that if $\Gamma$ is irreducible then so is $P_{\ep}$ and therefore it has a unique stationary distribution. This is implicitly assumed in the theorem statement above.
 \item  Note that taking $Q=\frac{1}{n}\1\1^{T}$ in Theorem~\ref{thm:general}, we recover Theorem~\ref{thm:Main}. 
 \item Note that since the effect of $Q$ appears only via $\Gamma$, in practice one can always assume that $Q$ has a block structure, namely, $Q(x, y)=\gamma_{ij}$ for all $x\in C_i, y\in C_j$. For instance, one can take $\gamma_{ij}=|C_j|^{-1}\Gamma(i, j)$. 
 \item This theorem, in some sense, characterizes all possible zero-noise limits for $P$. In principle, one can choose $\Gamma$ to achieve any target weights on the communicating classes of $P$. 

 \item Our model subsumes the more commonly used \textit{personalized PageRank} which corresponds to the choice $Q = \1\nu^T$ where the probability vector $\nu$ is the so called personalization vector.
 \end{enumerate}
 \end{remark}
 
In Section~\ref{sec:Proofs}, we prove the above two theorems. While Theorem~\ref{thm:Main} clearly follows from Theorem~\ref{thm:general}, we first present the proof of Theorem~\ref{thm:Main} which is simpler yet contains the key ideas of the proof of Theorem~\ref{thm:general}.  
 
\section{Proof of Theorem~\ref{thm:Main}}
\label{sec:Proofs}
Let $\G = (S,\E)$ denote the directed graph associated with our original reducible Markov chain with transition matrix $P$, where $\E:= \{(i,j) : P(i,j) > 0\}$ is the set of edges of $\G$.  Let $C_1, \cdots, C_M, M > 1,$ denote its closed communicating classes and $T$ the set of its transient states.  Let $C \coloneqq \cup_{i=1}^{M}C_i$ and $|C| = \sum_i|C_i|$. 
Let $P^\ep := (1-\ep)P + \frac{\ep}{N}\1\1^{T}, \; \ep > 0$.  This is an irreducible stochastic matrix under our assumptions.  Let $\pi^\ep$ denote its unique stationary distribution. For $\ep>0$, by the Markov chain tree theorem, we have 
$$\pi^\ep(i) = \frac{|H_i^{\epsilon}|}{\sum_j|H_j^{\epsilon}|}, \quad i \in S.$$
Here, as before, $H_i^{\epsilon}$ denotes the set of arborescences rooted at $i$ and $|H_i^{\epsilon}|$ its weight, with the $\epsilon$ dependence rendered explicit. Note that these are arborescences with respect to the irreducible transition matrix $P_{\epsilon}$. To understand the $\epsilon\to 0$ limit of $\pi^{\epsilon}$, we need to understand the structure the arborescences in $H_i^{\epsilon}$ for each $i$. This is what we do in the following.

Consider the reduced graph $\G_R$ obtained from the original graph $\G$ by compressing each communicating class $C_i$ into a single node $c_i$, called a \emph{reduced node}. Then the subgraph of $\G_R$ formed by the reduced nodes $\{c_i, 1 \leq i \leq M\}$ is a complete graph formed by $\ep$-edges (that is, edges with weight $\frac{\ep}{N}$), along with possibly some edges between transient states and from transient states to reduced nodes, that have weights $O(1)$. We refer to edges with $O(\ep)$ weight as $\ep$-edges. We have ignored the possibility of using the $\ep$-edges out of transient states. This will be justified later.

Consider a node $v\in S$. Let $H_v^{\epsilon} = \{g_1, \cdots , g_K\}$ be the set of arborescences rooted at $v$. For a reduced node $1\leq k\leq M$ and a node $v\in C_k$, let $H_v^{(k, \epsilon)}$ denote the set of arborescences in $C_k$ rooted at $v$, where the use of superscript $\epsilon$ is to emphasize that the edges of the arborescences are weighted by $P_{\epsilon}$. We allow, however, $\epsilon=0$ and we write $H_{v}^{(k)}=H_{v}^{(k, 0)}$ for the arborescences rooted at vertex $v\in C_k$ where the edges are weighted with our original reducible matrix $P$.  

For simplicity, fix $k=1$ and let $v\in C_1$. Consider an arborescence $g\in H_v^{\epsilon}$ rooted at a vertex $v$. 
An arborescence $g \in H_{v}^{\ep}$ comprises of:
\begin{enumerate}
\item an arborescence $h^{1} \in H_{v}^{(1, \epsilon)}$ of $C_1$ rooted at $v$, 
\item a spanning tree $X_{R}$ of $\G_R$ formed by $\ep$-edges,
\item arborescences $h^{\ell}$ in $C_{\ell}$ (rooted at some vertex in $C_\ell$) for each $2 \leq \ell \leq M$, and
\item edges from a spanning forest $T' \subset T$ that connect to each other (if they do) or to the $C_i$'s. By abuse of terminology,  we include the latter edges in $T'$. Define the weight  $|T'|$ of $T'$ as the product of the weights of edges in $T'$. 
\end{enumerate}

Let $g\in \G_0:= \G\backslash T$ be an arborescence rooted at some vertex $v\in C_1$ and let $h^{j}, 1\leq j\leq M$ and $X_R$ be constituents of the arboroscence $g$ as above. Then, naturally, the weight of $g$ is given by 
\[ |g|_{\epsilon} = |h^{1}|_{\epsilon}|X_{R}|_{\epsilon}\prod_{\ell_2=1}^{M}|h^{\ell}|_{\epsilon}\;,\]
where the subscript $\epsilon$ emphasizes the fact that the edges  weights are given by $P_{\epsilon}$. Therefore the weight $|g|_{\epsilon}$ is a polynomial in $\ep$. For this polynomial, it is easy to verify that the term with the smallest leading exponent in $\epsilon$ is given by the following proposition.

\begin{proposition}
\label{prop:Fristorder}
Let $g\in \G_0 := \G\backslash T$ be an arborescence as above. The weight  of $g$ is
$$|g|_{\epsilon}= \ep^{M-1}(1-\ep)^{|C| - M}|X_{R}|_{Q}|h^1|\Bigg(\prod_{\ell=2}^Mh^\ell\Bigg)+ O(\epsilon^M)\;.$$
Here $|h^{k}|$ denotes the weight of the arborescence $h^{k}$ with respect to the original matrix $P$ and $|X_{R}|_{Q}$ denotes the weight of $X_{R}$ where the edges are weighted according to the matrix $Q$. 
\end{proposition}

Observe that our inclusion of the trees that involve $\ep$-edges out of transient states would have only contributed terms that can be absorbed into the $O(\ep^M)$ term, therefore their presence does not affect the argument above. With this preparation, we are now ready to prove Theorem~\ref{thm:Main}.
\begin{proof}[Proof of Theorem~\ref{thm:Main}]
Note that for each arborescence of $\G_0$, an arborescence of $\G$ can be obtained by attaching to it a spanning forest $T'$ of $T$, thereby adding the multiplicative factor of $|T'|$ as above, for \textit{every} spanning forest. Let $|T|=\sum |T'|$, where the sum is over all distinct spanning forests $T'$ of $T$.  Let $\mathfrak{S}_M$ denote the set of all spanning trees of the complete graph on $M$ vertices $\{c_1,\ldots, c_M\}$. And, let $|\mathfrak{S}_M|\coloneqq \sum_{X_R\in \mathfrak{S}_M}|X_{R}|_{Q}$. Let us also define $|W_1|=\sum_{(h^2, \ldots, h^{M})}\prod_{\ell=2}^{M}|h^{\ell}|$ where the sum is over all $(h_2, \ldots, h_M)$ such that $h^{\ell}$ is an arborescence in $C_{\ell}$ (that is, the sum is over all possible roots in $C_{\ell}$ as well as all possible arborescences for each root). More generally, we define $|W_{k}| = \sum_{(h^1, \ldots, h^{M})}\prod_{\stackrel{\ell=1}{\ell\neq k}}^{M}|h^{\ell}|$ where the sum is over all $(M-1)$-tuples $(h^{1}, \ldots, h^{M})$ (where the $k$-th coordinate is missing) such that $h^{\ell}$ is an arborescence in $C_{\ell}$. Finally, we define $|W|=\sum_{(h^1, \ldots, h^{M})}\prod_{\ell=1}^{M}|h^{\ell}|$ where the sum is $M$-tuple $(h^1, \ldots, h^M)$ of arborescences such that $h^{\ell}\in C_{\ell}$.  

Let $v$ be some node in $C_1$. With this notation, we can write $|H_v^{\epsilon}|,$ the sum of the weights of arborescences rooted at $v$, as 
\begin{align*}
    |H_v^{\epsilon}| &= \ep^{M-1}(1-\ep)^{|C|-M}\left(\sum_{h^{1}\in H_v^{(1)}} |h^{1}|\right) |\mathfrak{S}_M||W_1||T| + O(\epsilon^M) \\
    &=\ep^{M-1}(1-\ep)^{|C|-M} |H_v^{(1)}| |\mathfrak{S}_M||W_1||T| + O(\epsilon^M)\;.
\end{align*}
Note that in the above computation, we took $v\in C_1$, but a similar computation can be carried out for a vertex $v\in C_k$ for any $1\leq k\leq M,$ thus yielding for $u\in C_k$
\[|H_v^{\epsilon}| =\ep^{M-1}(1-\ep)^{|C|-M} |H_v^{(k)}| |\mathfrak{S}_M||W_k||T| + O(\epsilon^M)\;.\]

Observe that for $u\in C_k$, we have
\begin{align*}
    \sum_{v\in C_k}|H_v^{\epsilon}| &=\ep^{M-1}(1-\ep)^{|C|-M} \left(\sum_{v\in C_k}|H_v^{(k)}|\right)|W_k| |\mathfrak{S}_M||T| + O(\epsilon^M)\\
    &= \ep^{M-1}(1-\ep)^{|C|-M}|W||\mathfrak{S}_M||T|+ O(\epsilon^M)\;,
\end{align*}
which is independent of $u$ and $k$ (at least in the leading order term). In particular, we have
\[\pi^{\epsilon}(C_k) = \frac{\sum_{v\in C_k}|H_v^{\epsilon}|}{\sum_{k=1}^{M}\sum_{v\in C_k}|H_v^{\epsilon}|}\to \frac{1}{M}\;,\]
as $\epsilon\to 0$. This proves the second part of Theorem~\ref{thm:Main}. To obtain the first part, we note that for a vertex $v\in C_k$, we have 
\begin{align*}
    \pi^{\epsilon}(v) &= \frac{|H_v^{\epsilon}|}{\sum_{k=1}^{M}\sum_{u\in C_k}|H_u^{\epsilon}|}= \frac{|H_v^{(k)}||W_k| + O(\epsilon)}{M|W|+ O(\epsilon)}\\
    &= \frac{|H_v^{(k)}||W_k| + O(\epsilon)}{M|W_k|\sum_{v\in C_k}|H_v^{(k)}|+ O(\epsilon)}= \frac{1}{M}\frac{|H_v^{(k)}|}{\sum_{v\in C_k}|H_v^{(k)}|} + O(\epsilon)\;.
\end{align*}
In the third equality, we use the fact that $|W|=|W_k|\sum_{v\in C_k}|H_v^{(k)}|$ for any $k\in \{1, \ldots, M\}$. The first claim in Theorem~\ref{thm:Main} is now immediate.

\end{proof}


The foregoing analysis shows that up to the zeroth order term in $\ep$, the value of $\ep > 0$ is irrelevant. To obtain the first order term in $\ep$ amounts to considering in the Markov chain tree theorem an arborescence rooted at $i$ (say) where exactly one other node in $T^c$ has an $\ep$-edge as the unique outgoing edge, chosen despite having at least one outgoing edge that is not an $\epsilon$-edge. Thus without this $\ep$-edge, the arborescence is a forest of two trees, one rooted at $i$ and one rooted at some other node in $T\backslash\{i\}$. Hence this term in the expansion of $\pi(i)$ has the sum of such weights of two-forests of trees with one rooted at $i$ and the other at some node in $T^c\backslash\{i\}$. One can also derive expressions for higher-order terms (in $\ep$) in this manner. It is clear that they will also be independent of $\ep$. Thus for $\ep$ sufficiently small in order to justify the Taylor series expansion for $\pi^\ep$, the coefficients of the expansion and therefore of $\pi^\ep$ itself are independent of $\ep$.

\section{Proof of Theorem~\ref{thm:general}}
\label{sec:Proof_generalization}
The proof of Theorem~\ref{thm:general} closely parallels the proof of Theorem~\ref{thm:Main}. To avoid repetition, we only point out the differences between the two proofs. Let $\G = (S,\E)$ denote the directed graph associated with our original reducible Markov chain with transition matrix $P$, where $\E:= \{(i,j) : P(i,j) > 0\}$ is the set of edges of $\G$.  Let $C_1, \cdots, C_M, M > 1,$ denote its closed communicating classes. For simplicity, we assume that there are no transient states. Let $S= C \coloneqq \cup_{i=1}^{M}C_i$ and $|C| = \sum_i|C_i|=n$. Let $P^\ep := (1-\ep)P + \ep Q, \; \ep > 0$. Since $\Gamma$ is irreducible, it follows that $P^{\epsilon}$ is irreducible. Let $\pi^{\ep}$ denote its unique stationary distribution which is given by 
\[\pi^{\ep}(v) = \frac{|H_v^{\ep}|}{\sum_{v\in S}|H_v^{\ep}|}\;,\]
by Markov chain tree theorem. Following the proof of Theorem~\ref{thm:Main}, we decompose an arborescence $g\in H_{v}^{\ep}$. For concreteness, we take $v\in C_1$. Consider an arborescence $g\in H_v^{\epsilon}$ rooted at a vertex $v$. 
An arborescence $g \in H_{v}^{\ep}$ comprises of:
\begin{enumerate}
\item an arborescence $h^{1} \in H_{v}^{(1, \epsilon)}$ of $C_1$ rooted at $v$, 
\item an arboroscence $X_{\Gamma}$ rooted at $C_1$ of $\G_{\Gamma}$ formed by $\ep$-edges whose vertices corresponding to $C_i$ are labeled by a vertex $x_i\in C_i$; and
\item arborescences $h^{\ell}$ in $C_{\ell}$ (rooted at some vertex in $C_\ell$) for each $2 \leq \ell \leq M$.
\end{enumerate}

Compared to the proof in Theorem~\ref{thm:Main}, the above decomposition has only two differences. First is that we assume that there are no transient states. This is purely for notational convenience. The key difference is that the $G_{\Gamma}$ is irreducible but need not be a complete graph, and an arborescence $X_{\Gamma}$ can have different weights depending on the labeling of its vertices. To emphasize the labeling, we write $X_{\Gamma}(x_1, \ldots, x_M)$ where $x_k\in C_k$. The weight of $g$ is given by 
\[ |g|_{\epsilon} = |h^{1}|_{\epsilon}|X_{R}(x_1, \ldots, x_{M})|_{\epsilon}\prod_{\ell_2=1}^{M}|h^{\ell}|_{\epsilon}\;,\]
where the subscript $\epsilon$ is to emphasize that the edges-weights are given by $P_{\epsilon}$. Therefore, the weight $|g|_{\epsilon}$ is a polynomial in $\ep$. Naturally, one obtains an analog of Proposition~\ref{prop:Fristorder} as follows.
\begin{proposition}
\label{prop:Fristorder_general}
Let $g\in \G$ be an arborescence as above. The weight  of $g$ is
$$|g|_{\epsilon}= \ep^{M-1}(1-\ep)^{|C| - M}|X_{\Gamma}(x_1, \ldots, x_{M})|_{Q}|h^1|\Bigg(\prod_{\ell=2}^Mh^\ell\Bigg)+ O(\epsilon^M)\;.$$
Here $|h^{k}|$ denotes the weight of the arborescence $h^{k}$ with respect to the original matrix $P$ and $|X_{R}(x_1, \ldots, x_{M})|_{Q}$ denotes the weight of $X_{\Gamma}$ where the edges are weighted according to matrix $Q$. 
\end{proposition} 


\begin{proof}[Proof of Theorem~\ref{thm:general}]
We follow the proof of Theorem~\ref{thm:Main} with some differences that we point out below. Let $H^{\Gamma}_k$ denotes the set of all arborescences rooted at $k$ in $\Gcal_{\Gamma}$. It is easy to verify that for $h\in H^{\Gamma}_k$, we have 
\begin{align*}
    |h|_{\Gamma} = \prod_{(u, v)\in h}\Gamma(u, v)= \sum_{(x_1, \ldots, x_M)} |h(x_1, \ldots, x_{M})|= \sum_{(x_1, \ldots, x_M)}(\prod_k|C_k|)^{-1}\prod_{(u, v)\in h}Q(x_u, x_v)\;,
\end{align*}
where the last two sums are over all possible labels $(x_1, \ldots, x_M)\in C_1\times \cdots\times C_M$ of the \emph{skeleton} $h$ such that the vertex $i$ in $h$ is labeled by $x_i\in C_i$. Let us define $ |\mathfrak{S}_{\Gamma, k}|\coloneqq \sum_{h\in H^{\Gamma}_k}|h|$ and $|\mathfrak{S}|_{\Gamma}=\sum_{k}|\mathfrak{S}_{\Gamma, k}|$.  Let $\pi_{\Gamma}$ be the unique stationary distribution of $\Gamma$. By Markov chain tree theorem we obtain 
\[\pi_{\Gamma}(k) = \frac{|H^{\Gamma}_k|}{\sum_{k=1}^{M}|H^{\Gamma}_{k}|}=\frac{|\mathfrak{S}_{\Gamma, k}|}{|\mathfrak{S}_{\Gamma}|} \;.\]
Here again, we  remind the reader that this step was not necessary in the proof of Theorem~\ref{thm:Main} because the quantity $|\mathfrak{S}_{\Gamma, k}|$ is independent of $k$ by virtue of the choice of $Q=\frac{1}{n}\bf 1\bf 1^{T}$.  

With this observation, the proof of Theorem~\ref{thm:Main} can be readily adapted. We skip the details for brevity and only give the outline of the remaining proof. With the above setup, it is easily verified that for any $v\in C_k$, 
\[|H_v^{\epsilon}| =\ep^{M-1}(1-\ep)^{|C|-M} |H_v^{(k)}| |\mathfrak{S}_{\Gamma, k}||W_k| + O(\epsilon^M)\;,\]
where 
$|W_{k}| = \sum_{(h^1, \ldots, h^{M})}\prod_{\stackrel{\ell=1}{\ell\neq k}}^{M}|h^{\ell}|$ is as in the proof ofTheorem~\ref{thm:Main}. Similarly, we also define $|W|=\sum_{(h^1, \ldots, h^{M})}\prod_{\ell=1}^{M}|h^{\ell}|$ as in the proof of Theorem~\ref{thm:Main}. To finish the proof, we observe that 
\begin{align*}
    \pi^{\ep}(v) &=\frac{|H_v^{\ep}|}{\sum_{v}|H_v|^{\ep}}=\frac{|H_v^{(k)}| |\mathfrak{S}_{\Gamma, k}||W_k|}{\sum_{k=1}^{M}\sum_{v\in C_{k}}|H_v^{(k)}| |\mathfrak{S}_{\Gamma, k}||W_k|}+O(\ep)\\
    &=\frac{|H_v^{(k)}| |\mathfrak{S}_{\Gamma, k}||W_k|}{|W||\mathfrak{S}_{\Gamma}|}+ O(\ep)\\
    &=\frac{|H_v^{(k)}|}{\sum_{v\in C_k}|H_v^{(k)}|}\frac{|\mathfrak{S}_{\Gamma, k}|}{|\mathfrak{S}_{\Gamma}|}+ O(\ep)\;,
\end{align*}
where we used in the last line the fact that $|W|=\sum_{v\in C_k}|H_v^{(k)}|\times |W_k|$ for any $k$. Let $\pi_k$ be the unique stationary distribution of $P$ supported on $C_k$. Then, we conclude that 
\[ \pi^{\ep}(v) = \pi_{k}(v)\pi_{\Gamma}(k)+ O(\ep)\;, \quad \forall v\in C_k, \;\; k\in \{1, \ldots, M\}\;.\]
Summing over all $v\in C_k$, we also obtain $\pi^{\ep}(C_k)=\pi_{\Gamma}(k)+ O(\ep)$, completing the proof.

\end{proof}

\section{Examples}
\label{sec:Examples}
\noindent We give a few illustrative examples.
\begin{example}
Let $d_i :=$ the outdegree of $i$. If we use the simple random walk as proposed in the original PageRank model, the weight of an arborescence rooted at $i$ is $\prod_{j\neq i}\frac{1}{d_j} \propto d_i\times\prod_{j\in S} \frac{1}{d_j}$. Hence the number of arborescences rooted at $i$ is $\propto \frac{\pi(i)}{d_i}$, i.e.\ the PageRank of $i$ is proportional to the product of the degree of $i$ and the number of arborescences rooted at $i$. This computation was used in \cite{KBK} to motivate $\frac{\pi(i)}{d_i}$, which is proportional to the number of arborescences rooted at $i$, as a measure for rumor source detection.
\end{example}

\begin{example}
 A more general model adapted from the Bradley-Terry model  \cite{Bradley} postulates a weight of popularity $w_i$ for node $i$ for $i\in S$. Let $\N(i) := \{k \in S : (i,k)$ is an edge in $\G_R\}$, i.e., the set of successors of $i$ in $\G_R$. Consider a stochastic matrix with $(i,j)$th element $= \frac{w_j}{W(i)}$ for $W(i) := \sum_{k\in\N(i)}w_k$. For an arborescence $A$ rooted at $i$, let $L(A):=$ the set of its leaves, i.e., the nodes that do not have any incoming edge. Then a calculation analogous to the one above shows that $\pi(i)$ is proportional to $\sum_{A\in H_i(A)}\frac{W(i)}{\prod_{j\in L(A)}w(j)}$. The $w_i$'s can be estimated from user data. 
\end{example}

\begin{example}
    Another model based on \cite{Bradley} is that of \cite{Shah} where we assign weights to edges instead of nodes. Thus $w_{ij}$ is the weight assigned to $(i,j) \in \E$. Intuitively this indicates the  `fraction of people who prefer $j$ over $i$. Set
\[p(i,j) = \frac{1}{D}\left(\frac{w_{ij}}{w_{ij}+ w_{ji}}\right)\]
for $i \neq j, (i,j) \in \E$, where $D:=$ the maximum outdegree of nodes in $S$, and $p(i,i) = 1 - \sum_{j\neq i}p(i,j)$ for $i\in S$. This too can be estimated from data \cite{Shah}.
\end{example}

Other parametric models for edge probabilities are possible.

\section{Concluding remarks}

It should be kept in mind that the foregoing is still in the framework of the entire web graph. Clearly, the search is modulated by the initial keywords provided. While the add-ons this calls for are quite sophisticated, as a first cut we may assume that the keywords help restrict the search to the subgraph formed by the most relevant nodes. Then we can restrict the foregoing analysis to this subgraph. 

\subsection*{Acknowledgements}
The work was done in part during the Workshop on the Mathematical Foundations of Network Models and Their Applications, as part of the BIRS-CMI pilot program. The workshop was held at the Chennai Mathematical Institute (CMI) from December 15, 2024, to December 20, 2024, and a Research School was conducted in the week preceding the BIRS event (December 9-13). The events were supported by Banff International Research Station (BIRS), Chennai Mathematical Institute (CMI), and the National Board for Higher Mathematics (NBHM). The work of VSB was supported in part by a grant from Google Research Asia. VSB thanks Prof.\ Soumen Chakrabarti for helpful discussions.\\

\noindent \textbf{\large COMMENT :} It was brought to our notice that results equivalent to ours are available in resp.,\\

\noindent Avrachenkov, K., Litvak, N., 2004. Decomposition of the Google PageRank
and optimal linking strategy, INRIA Research Report No.\ 5101.\\

\noindent Avrachenkov, K., Litvak, N.\ and Pham, K.S., 2008. A singular perturbation approach for choosing the PageRank damping factor. Internet Mathematics, 5(1-2), pp.47-69.\\

Nevertheless, the proofs here based on the Markov chain tree theorem may be of independent interest. We thank Dr.\ Konstantin Avrachenkov for pointing out these references.\\

\bibliographystyle{alpha} 


\end{document}